\documentclass[11pt]{amsart}
\usepackage{amsmath, amssymb, amsthm}
\usepackage[all]{xy}
\usepackage{pgf, tikz}
\usetikzlibrary{arrows, positioning, calc, chains}
\tikzset{
	ch/.style={circle,draw,on chain,inner sep=2pt},
	chj/.style={ch,join},
	every path/.style={shorten >=4pt,shorten <=4pt}
}

\newcommand{\RomanNumeralCaps}[1]
{\MakeUppercase{\romannumeral #1}}
\usepackage[colorlinks=true, pdfstartview=FitV, linkcolor=blue,%
citecolor=blue, urlcolor=blue]{hyperref}
\newtheorem{Def}{Definition}[section]
\newtheorem{cor}[Def]{Corollary}

\newtheorem{thm}{Theorem} [section]
\newtheorem{ex}[thm]{Example}
\newtheorem{rem}[thm]{Remark}
\newtheorem{prop}{Proposition}[section]
\numberwithin{equation}{section}
\title[Differential operator approach to
$\imath$-quantum groups ]{Differential operator approach to
 	$\imath$quantum groups and their
 	Oscillator representations}

\author[Zhaobing Fan]{Zhaobing Fan}
\address{Harbin Engineering University, Harbin, China}
\email{fanzhaobing@hrbeu.edu.cn}

\author[Jicheng Geng]{Jicheng Geng}
\address{Harbin Engineering University, Harbin, China}
\email{jcgeng@hrbeu.edu.cn}

\author[Shaolong Han]{Shaolong Han}
\address{Harbin Engineering University, Harbin, China}
\email{algebra@hrbeu.edu.cn}


\begin{document}
\begin{abstract}
For a quasi-split Satake diagram, we define a modified $q$-Weyl algebra, and show that there is an algebra homomorphism between it and the corresponding $\imath$quantum group.
In other words, we provide a differential operator approach to $\imath$quantum groups.
Meanwhile, the oscillator representations of $\imath$quantum groups are obtained.
The crystal basis of the irreducible subrepresentations of these oscillator representations are constructed.
\end{abstract}
\maketitle

\section{Introduction}
Quantum groups were introduced independently by Drinfeld and Jimbo in \cite{Dr86,J86}.
To date, the theory of quantum groups has played important roles in representation theory, low dimensional topology and mathematical physics etc.
From the perspective of the development of quantum groups, mathematicians have extended various theories of quantum groups to more and more generalized cartan data, such as from semi-simple cases to Kac-Moody cases, and then to Borcherds cases.
Another direction to develop quantum groups is the so-called $\imath$quantum groups, which are the main objects studied in the current paper.

Let $(\mathfrak g, \mathfrak g^{\theta})$ be a symmetric pair, where $\mathfrak g$ is a Kac-Moody algebra, $\theta$ is an involution of $\mathfrak g$ and $\mathfrak g^{\theta}$ is the $\theta$-stable subalgebra of $\mathfrak g$.
	The $\imath$quantum groups are coideal subalgebras of $\mathbf U_q(\mathfrak g)$, which specializes to $\mathbf U(\mathfrak g^\theta)$ at $q\to 1$ and denoted by $\mathbf U'_q(\mathfrak g^\theta)$.
In general, the algebra $\mathbf U_q(\mathfrak g^{\theta})$ is not a Hopf algebra.
The $\imath$quantum groups are classified by Satake diagrams. 
Moreover, $\imath$quantum groups are isomorphic to ordinary quantum groups if the Satake diagram is quasi-split of diagonal type.

In 2013, Bao and Wang initiated the $\imath$-program in \cite{BW18a}.
The purpose of this program is to generalize various results for ordinary quantum groups to $\imath$quantum groups.
From geometric side, in \cite{BKLW}, Bao, Kujawa, Li and Wang gave a geometric realization of $\mathbf U^{\jmath}$ and $\mathbf U^{\imath}$, the $\imath$quantum groups of type   $B/C$, by using equivariant function on double flag varieties.
This generalizes the classic work of
 Beilinson, Lusztig and MacPherson for ordinary quantum group in \cite{BLM}.
 Later on, the first author and his collaborators further gave a geometric realization of the $\imath$quantum groups of affine type $C$ in \cite{FLLLW20} by using a similar approach. 
 Moreover, the first author, Ma and Xiao provided another geometric realization of the $\imath$quantum groups $\mathbf U^{\jmath}$ in \cite{FMX} by using equivariant $K$-group of corresponding Steinberg variety.
 This fills Ginzburg and Vasserot's work in \cite{GV93} into $\imath$-program.

From algebraic side, $q$-Schur duality, Hall algebra and categorification are important theory related to quantum groups.
In the framework of $\imath$-program, these theories have been generalized to certain $\imath$quantum groups.
The Hecke-algebraic approach to the $\imath$quantum groups was given in \cite{FLLLWb}, which was a generalization of affine type $A$ construction in \cite{DF15}.
The three-parameter $q$-Schur duality between $\imath$quantum groups and Hecke algebras was studied in \cite{FLLLWW},
which can be degenerated to $q$-Schur dualities of affine types $B$, $C$ and $D$ by various specialization of parameters.
In \cite{R90, Br13},   Ringel and  Bridgeland gave a realization of the half part and entire of quantum groups by using Hall algebras, respectively.
The $\imath$Hall algebra approach toward $\imath$quantum groups was developed by Lu and Wang in series papers \cite{LW19a,LW21a,LW19b,LW20}.
 The highest weight theory and crystal (global) basis theory of $\imath$quantum groups $\mathbf U^{\jmath}$  was studied in \cite{W20} and \cite{W21}.
 The categorification of the modified  $\imath$quantum group $\dot{\mathbf U}^{\jmath}$ was developed in \cite{BSWW}.
In \cite{DW20,DW21}, Du and Wu presented a new realization of the $\imath$quantum groups $\mathbf U^\jmath$ and $\mathbf U^{\imath}$ in terms of a BLM type basis.

 In \cite{H}, Hayashi used differential operators on polynomial ring to construct spinor and oscillator representations for quantum groups of types  $A_{N-1}$, $B_N$, $C_N$, $D_N$ and $A_{N-1}^{(1)}$.
	In \cite{H00, ZH, DZ, DLZ}, Hu and Du further studied the differential operator realization  for quantum groups $U_q(\mathfrak{sl}_n)$, $U_q(\mathfrak{sp}_{2n})$, $U_q(\mathfrak{gl}_{m|n})$ and $U_q(\mathfrak {q}_n)$	 respectively.

In this paper, we provide a differential operator realization of 
quasi-split finite and affine
$\imath$quantum groups $ ^{\imath}U(\mathcal S)$, 
where $\mathcal S$ is the Satake diagram of types $A_r$ or $A_r^{(1)}$with nontrivial diagram involutions.

This realization can be regarded as an application of Hayashi's approach \cite{H} to $\imath$quantum groups.
To achieve this goal, we introduce the modified $q$-Weyl algebra $\mathbf{A}_q(\mathcal S)$  associated to the Satake diagrams $\mathcal S$, generated by certain linear operators on the polynomial ring $\mathbb Q(q)[X_0,X_1,\cdots,X_{r+1}]$,
and show that there is an algebra homomorphism from $^{\imath}U(\mathcal S)$ to $\mathbf{A}_q(\mathcal S)$.
Thus we obtain naturally a differential operator realization of $^{\imath}U(\mathcal S)$ on the $q$-Weyl algebra $\mathbf A_q(A_{r+1})$ and representations of $^{\imath}U(\mathcal S)$ on the polynomial ring which are called the oscillator representations.
Following the framework of crystal basis theory in \cite{Ka91} and \cite[Chapter 4]{HK}, we construct the crystal basis of irreducible modules arising from the oscillator representations of  $\imath$quantum groups $^{\imath}U(A_{2r+1})$, $^{\imath} U(A_{2r+1}^{(1)})$ and $^{\imath} U(A_{1}^{(1)})$.

The paper is organized as follows.
In Section \ref{Quantum groups}, we recall the differential operator realization of the quantum group $\mathbf U_q(\mathfrak{sl}_{r+2})$ and its oscillator representation. In Section \ref{iWeyl algebra}, we define the modified $q$-Weyl algebas $\mathbf{A}_q(\mathcal S)$ associated with Satake diagrams and study their representations. In Section \ref{Differential operator realization}, we construct the differential operator realizations of $\imath$quantum groups $^{\imath}U(\mathcal S)$. In Section \ref{Crystal basis}, we study the oscillator representations of $^{\imath}U(\mathcal S)$ and the crystal basis of irreducible modules over $\imath$quantum groups $^{\imath}U(A_{2r+1})$, $^{\imath} U(A_{2r+1}^{(1)})$ and $^{\imath} U(A_{1}^{(1)})$.

{\bf Acknowledgements.}
Z. Fan was partially supported by the NSF of China grant 11671108, the NSF of Heilongjiang Province grant JQ2020A001, and the Fundamental Research Funds for the central universities.

{\bf Notations.}
 Let
\begin{align*}
&[a]=\dfrac{q^a-q^{-a}}{q-q^{-1}},\quad {[a]}^k!=[ka][k(a-1)]\cdots[2k][k],\quad X^{{(a)}_k}=\dfrac{X^a}{{[a]}^k!}\ \ (a,k\in\mathbb Z),\\
&	(a;x)_0=1,\quad (a;x)_n=(1-a)(1-ax)\cdots(1-ax^{n-1}),\quad n\geq 1,\\
& \qquad\qquad\qquad\begin{bmatrix} n\\d \end{bmatrix} =
\begin{cases}
	\frac{[n][n-1]\ldots [n-d+1]}{{[d]}!}, & \text{ if }d > 0,
	\\
	1,  &\text{ if }d = 0,\\
	0, & \text{ if }d<0.
\end{cases}
\end{align*}

\section{Differential operator realization of $\mathbf U_q(\mathfrak{sl}_{r+2})$}\label{Quantum groups}
Let $r\ge 0$ and $\mathbb I=\{0,1,\cdots,r+1\}$.
The $q$-Weyl algebra $\mathbf{A}_q(A_{r+1})$ associated to Dynkin diagram of type $A_{r+1}$ is generated by $\mathcal D_i$, $\mathcal X_i$, $\mathcal M_i^{\pm 1}$ ($i\in\mathbb I$) over $\mathbb Q(q)$ subject to the following relations (see \cite{H}):
\begin{align}
	&\mathcal M_i\mathcal M_i^{-1}=\mathcal M_i^{-1}\mathcal M_i=1,\
	\mathcal M_i\mathcal M_j=\mathcal M_j\mathcal M_i,\label{MiMj}\\
	&\mathcal D_i\mathcal M_j=\mathcal M_j\mathcal D_i,\
	\mathcal X_i\mathcal M_j=\mathcal M_j\mathcal X_i,\
	\mathcal D_i\mathcal X_j=\mathcal X_j\mathcal D_i,\ \ \text{if $i\neq j$}, \\	
	&\mathcal D_i\mathcal D_j=\mathcal D_j\mathcal D_i,\ \mathcal X_i\mathcal X_j=\mathcal X_j\mathcal X_i,\\
	&\mathcal D_i\mathcal M_i=q\mathcal M_i\mathcal D_i,\
	\mathcal X_i\mathcal M_i=q^{-1}\mathcal M_i\mathcal X_i,\label{DiMiXiMi}\\
	&\mathcal{D}_i\mathcal{X}_i=\frac{q\mathcal M_i-q^{-1}\mathcal M_i^{-1}}{q-q^{-1}},\ \mathcal{X}_i\mathcal{D}_i=\frac{\mathcal M_i-\mathcal M_i^{-1}}{q-q^{-1}}.\label{DiXiXiDi}
\end{align}

Let $\mathbb P=\mathbb Q(q)[X_0,X_1,\cdots,X_{r+1}]$ be the polynomial ring over $\mathbb Q(q)$.
We define $\mathbf{A}_q(A_{r+1})$ acting on $\mathbb P$, for any $f\in \mathbb P$, as follows
\begin{equation}\label{q-differentiation operator}	
	\mathcal D_if(X_0,\cdots, X_{r+1})
	=\frac{f(X_0,\cdots,qX_i,\cdots,X_{r+1})-f(X_0,\cdots,q^{-1}X_i,\cdots,X_{r+1})}{qX_i-q^{-1}X_i},
\end{equation}
\begin{equation*}	
	\mathcal  X_if(X_0,\cdots,X_i,\cdots,X_{r+1})=X_if(X_0,\cdots,X_i,\cdots,X_{r+1}),
\end{equation*}
\begin{equation*}	
	\mathcal M_i^{\pm1}f(X_0,\cdots,X_i,\cdots,X_{r+1})=f(X_0,\cdots,q^{\pm1}X_i,\cdots,X_{r+1}).
\end{equation*}

\begin{thm}
[{\cite[Proposition 2.1]{H}}]\label{irreducible Aq module}
	The polynomial ring $\mathbb P$ is an irreducible $\mathbf{A}_q(A_{r+1})$-module.
\end{thm}

The generators $\mathcal D_i$, $\mathcal X_i$, $\mathcal M_i$ ($i\in\mathbb I$) of the algebra   $\mathbf{A}_q(A_{r+1})$ can be viewed as linear operators on the polynomial ring $\mathbb P$. Moreover,
the operator $\mathcal D_i$ is called \textit{$q$-differentiation operator}, which satisfies
 the $q$-analog of the Leibniz rule as follows:
\begin{align*}
	&\mathcal D_i(f(X_0,\cdots,X_i,\cdots,X_{r+1})g(X_0,\cdots,X_i,\cdots,X_{r+1}))\\
	=&\mathcal D_if(X_0,\cdots,X_i,\cdots,X_{r+1})g(X_0,\cdots,q^{-1}X_i,\cdots,X_{r+1})\\
	&\qquad\qquad\qquad\qquad\qquad+f(X_0,\cdots,qX_i,\cdots,X_{r+1})\mathcal D_ig(X_0,\cdots,X_i,\cdots,X_{r+1}).
\end{align*}

 Let $\texttt{c}_{ij}=2\delta_{ij}-\delta_{i,j+1}-\delta_{i+1,j}$ for $i,j\in\mathbb Z$.
  Recall that the quantum group $\mathbf U_q(\mathfrak{sl}_{r+2})$ is generated by $E_i,~F_i,~K_i^{\pm1}$ ($0\leq i\leq r$)  satisfying the following relations:
\begin{align*}
	&K_iK_i^{-1}=K_i^{-1}K_i=1,K_iK_j=K_jK_i,\\
	&K_iE_jK_i^{-1}=q^{\texttt{c}_{ij}}E_j,~K_iF_jK_i^{-1}=q^{-\texttt{c}_{ij}}F_j,\\
	&E_iF_j-F_jE_i=\delta_{ij}\frac{K_i-K^{-1}_i}{q-q^{-1}},\\
	&	E_{i} E_{j} = E_{j} E_{i},~ F_{i} F_{j} = F_{j}  F_{i},\ \ \text{if $|i-j|>1$}, \\	
	&	E_{i}^2 E_{j}  - (q+q^{-1}) E_{i} E_{j} E_{i}+E_{j} E_{i}^2=0,\ \ \text{if $|i-j|=1$} ,\\
	&	F_{i}^2 F_{j}  - (q+q^{-1})  F_{i} F_{j} F_{i}+F_{j} F_{i}^2=0,\ \ \text{if $|i-j|=1$} .
\end{align*}

\begin{thm}[{\cite[Theorem 3.2]{H}}]\label{chi_r}
	There is a $\mathbb Q(q)$-algebra
	homomorphism $\chi_r:\mathbf{U}_q(\mathfrak {sl}_{r+2})\rightarrow \mathbf{A}_q(A_{r+1})$ such that
	\begin{equation*}
		E_i\mapsto\mathcal{X}_{i}\mathcal{D}_{i+1},\quad
		F_i\mapsto\mathcal{X}_{i+1}\mathcal{D}_{i},\quad
		K_i\mapsto \mathcal{M}_i \mathcal{M}_{i+1}^{-1}.
	\end{equation*}	
\end{thm}
By Theorems \ref{irreducible Aq module} and \ref{chi_r},  the polynomial ring $\mathbb P$ is a $\mathbf{U}_q(\mathfrak {sl}_{r+2})$-module via pull-back map. The  polynomial ring $\mathbb P$ is called the oscillator representation of $\mathbf{U}_q( \mathfrak {sl}_{r+2})$. Moreover, we have
\begin{thm}[{\cite[Theorem 4.1]{H}}]
	Let $s\ge 0$. Then the subspace of $\mathbb P$  spanned by the set $$\{X_0^{a_0}X_1^{a_1}\cdots X_{r+1}^{a_{r+1}}|a_0+a_1+\cdots+a_{r+1}=s\}$$
	is an irreducible $\mathbf{U}_q(\mathfrak {sl}_{r+2})$-module.
\end{thm}

\section{Modified $q$-Weyl algebra}\label{iWeyl algebra}
In this section, we shall introduce the modified $q$-Weyl algebra associated to a given Satake diagram.
Let $\mathbf I$ be a finite set and assume $|\mathbf I|\ge r+2$. Denote by $D$ the Dynkin diagram of type $A_{|\mathbf I|-1}$ or $A_{|\mathbf I|-1}^{(1)}$. 
The nodes in $D$ are labeled by the elements in $\mathbf I$. For $i,j\in\mathbf I$,  we define the Cartan datum ($\mathbf I$, $\cdot$) on $D$ as follows:
\begin{equation*}
	i\cdot j=
	\begin{cases}
		2,\quad \text{if $i=j$},\\
		-k,\quad \text{if there are $k$ lines connecting $i$ and $j$},\\
		0,\quad \text{otherwise}.
	\end{cases}
\end{equation*}

Let $(Y,X, \langle\cdot,\cdot\rangle)$ be a root datum of type $(\mathbf I, \cdot)$; cf. \cite[Section 2.1]{Lus93}. Let $\tau$ be an involution of the Cartan
datum ($\mathbf I$, $\cdot$) and $W$ be the Weyl group generated
by simple reflections $s_i$ for $i\in\mathbf I$. Given a subset $\mathbf I_{\bullet}\subset \mathbf I$, let $W_{\mathbf I_\bullet} =\langle s_i \mid i\in \mathbf I_{\bullet}\rangle$ be the parabolic subgroup of $W$ with longest element $w_{\bullet}$, and let $\rho_{\bullet}$ (resp. $\rho^{\vee}_{\bullet}$) be the half sum of all positive roots (resp. coroots) in the root system $R_{\mathbf I_\bullet}$ (resp. $R^{\vee}_{\mathbf I_\bullet}$).

Let $\mathbf I_{\circ} = \mathbf I \backslash \mathbf I_{\bullet}$. A Satake diagram $\mathcal S$ is the pair ($\mathbf I=\mathbf I_{\circ}\cup\mathbf I_{\bullet},\tau$) such that $\tau (\mathbf I_{\bullet}) = \mathbf I_{\bullet}$, the actions of $\tau$ and $-w_{\bullet}$ on $\mathbf I_{\bullet}$ coincide, and
$\langle \rho^\vee_{\bullet}, j' \rangle \in \mathbb Z$ if $\tau j =j \in \mathbf I_{\circ}$.
By abuse of notations, denote by $\mathbb I$ the set of $\tau$-orbits of $\mathbf I$.

 In this paper, we only consider the case $\mathbf I=\mathbf I_{\circ}$. We use natural numbers $\{0,1,2,\cdots\}$ to label the elements in $\mathbf I$.

\begin{Def}\label{modified q-Weyl algebra}
Given a Satake diagram $\mathcal S$, the associated modified $q$-Weyl algebra, $\mathbf{A}_q(\mathcal S)$, is generated by $\mathfrak d_i$, $\mathfrak x_i$, $\mathfrak m_i$ ($i\in \mathbb I$) over $\mathbb{Q}(q)$  subject to the following relations:
\begin{align}
	&\mathfrak m_i\mathfrak m_i^{-1}=\mathfrak m_i^{-1}\mathfrak m_i=1,\
	\mathfrak m_i\mathfrak m_j=\mathfrak m_j\mathfrak m_i,\label{mimj}\\
	&\mathfrak d_i\mathfrak m_j=\mathfrak m_j\mathfrak d_i,\
	\mathfrak x_i\mathfrak m_j=\mathfrak m_j\mathfrak x_i,\
	\mathfrak d_i\mathfrak x_j=\mathfrak x_j\mathfrak d_i,\ \ \text{if $i\neq j$},\\	
	&\mathfrak d_i\mathfrak d_j=\mathfrak d_j\mathfrak d_i,\ \mathfrak x_i\mathfrak x_j=\mathfrak x_j\mathfrak x_i,\\
	&\mathfrak d_i\mathfrak m_i=q^{\xi_i}\mathfrak m_i\mathfrak d_i,\
	\mathfrak x_i\mathfrak m_i=q^{-\xi_i}\mathfrak m_i\mathfrak x_i,\label{dimiximi}\\
	&\mathfrak{d}_i\mathfrak{x}_i=\frac{q^{\xi_i}\mathfrak m_i-q^{-\xi_i}\mathfrak m_i^{-1}}{q-q^{-1}},\ \mathfrak{x}_i\mathfrak{d}_i=\frac{\mathfrak m_i-\mathfrak m_i^{-1}}{q-q^{-1}},\label{dixixidi}
\end{align}
where $\xi_i=1-i\cdot\tau i$.
\end{Def}

In the rest of the paper, we shall consider the following Satake diagrams.

\begin{center}	
	\label{figure:j}	
	\begin{tikzpicture}
		\coordinate (A) at (0,1.5);
		\node at (A) {Diagram \RomanNumeralCaps{1}
			: Satake diagram of type $A_{2r+3}$.};
		\matrix [column sep={0.6cm}, row sep={0.5 cm,between origins}, nodes={draw = none,  inner sep = 3pt}]
		{
			\node(U1) [draw, circle, fill=white, scale=0.6, label = 0] {};
			&\node(U2)[draw, circle, fill=white, scale=0.6, label =1] {};
			&\node(U3) {$\cdots$};
			&\node(U4)[draw, circle, fill=white, scale=0.6, label =$r$] {};
			&\node(U5)[draw, circle, fill=white, scale=0.6, label =$r+1$] {};
			\\
			&&&&&
			\\
			\node(L1) [draw, circle, fill=white, scale=0.6, label =below:$2r+3$] {};
			&\node(L2)[draw, circle, fill=white, scale=0.6, label =below:$2r+2$] {};
			&\node(L3) {$\cdots$};
			&\node(L4)[draw, circle, fill=white, scale=0.6, label =below:$r+3$] {};
			&\node(L5)[draw, circle, fill=white, scale=0.6, label =below:$r+2$] {};
			\\
		};
		\begin{scope}
			
			\draw (U1) -- node  {} (U2);
			\draw (U2) -- node  {} (U3);
			\draw (U3) -- node  {} (U4);
			\draw (U4) -- node  {} (U5);
			\draw (U5) -- node  {} (L5);
			\draw (L1) -- node  {} (L2);
			\draw (L2) -- node  {} (L3);
			\draw (L3) -- node  {} (L4);
			\draw (L4) -- node  {} (L5);
			\draw (L1) edge [color = blue,<->, bend right, shorten >=4pt, shorten <=4pt] node  {} (U1);
			\draw (L2) edge [color = blue,<->, bend right, shorten >=4pt, shorten <=4pt] node  {} (U2);
			\draw (L4) edge [color = blue,<->, bend left, shorten >=4pt, shorten <=4pt] node  {} (U4);
			\draw (L5) edge [color = blue,<->, bend left, shorten >=4pt, shorten <=4pt] node  {} (U5);
		\end{scope}
	\end{tikzpicture}
\end{center}

\begin{center}
	\label{figure:i}
	\begin{tikzpicture}
		\coordinate (A) at (0,1.5);
		\node at (A) {Diagram \RomanNumeralCaps{2}: Satake diagram of type $A_{2r+2,1}$.};		
		\matrix [column sep={0.6cm}, row sep={0.5 cm,between origins}, nodes={draw = none,  inner sep = 3pt}]
		{
			\node(U1) [draw, circle, fill=white, scale=0.6, label = 0] {};
			&\node(U2)[draw, circle, fill=white, scale=0.6, label =1] {};
			&\node(U3) {$\cdots$};
			&\node(U5)[draw, circle, fill=white, scale=0.6, label =$r$] {};
			\\
			&&&&
			\node(R)[draw, circle, fill=white, scale=0.6, label =$r+1$] {};
			\\
			\node(L1) [draw, circle, fill=white, scale=0.6, label =below:$2r+2$] {};
			&\node(L2)[draw, circle, fill=white, scale=0.6, label =below:$2r+1$] {};
			&\node(L3) {$\cdots$};
			&\node(L5)[draw, circle, fill=white, scale=0.6, label =below:$r+2$] {};
			\\
		};
		\begin{scope}
			\draw (U1) -- node  {} (U2);
			\draw (U2) -- node  {} (U3);
			\draw (U3) -- node  {} (U5);
			\draw (U5) -- node  {} (R);
			
			\draw (L1) -- node  {} (L2);
			\draw (L2) -- node  {} (L3);
			\draw (L3) -- node  {} (L5);
			\draw (L5) -- node  {} (R);
			\draw (R) edge [color = blue,loop right, looseness=40, <->, shorten >=4pt, shorten <=4pt] node {} (R);
			\draw (L1) edge [color = blue,<->, bend right, shorten >=4pt, shorten <=4pt] node  {} (U1);
			\draw (L2) edge [color = blue,<->, bend right, shorten >=4pt, shorten <=4pt] node  {} (U2);
			\draw (L5) edge [color = blue,<->, bend left, shorten >=4pt, shorten <=4pt] node  {} (U5);
		\end{scope}
	\end{tikzpicture}
\end{center}
\begin{center}
	\begin{tikzpicture}
		\coordinate (A) at (0,1.5);
		\node at (A) {Diagram \RomanNumeralCaps{3}: Satake diagram of type $A^{(1)}_{2r+3}$.};
		\matrix [column sep={0.6cm}, row sep={0.5 cm,between origins}, nodes={draw = none,  inner sep = 3pt}]
		{
			\node(U1) [draw, circle, fill=white, scale=0.6, label = 0] {};
			&\node(U2)[draw, circle, fill=white, scale=0.6, label =1] {};
			&\node(U3) {$\cdots$};
			&\node(U4)[draw, circle, fill=white, scale=0.6, label =$r$] {};
			&\node(U5)[draw, circle, fill=white, scale=0.6, label =$r+1$] {};
			\\
			&&&&&
			\\
			\node(L1) [draw, circle, fill=white, scale=0.6, label =below:$2r+3$] {};
			&\node(L2)[draw, circle, fill=white, scale=0.6, label =below:$2r+2$] {};
			&\node(L3) {$\cdots$};
			&\node(L4)[draw, circle, fill=white, scale=0.6, label =below:$r+3$] {};
			&\node(L5)[draw, circle, fill=white, scale=0.6, label =below:$r+2$] {};
			\\
		};
		\begin{scope}
			\draw (L1) -- node  {} (U1);
			\draw (U1) -- node  {} (U2);
			\draw (U2) -- node  {} (U3);
			\draw (U3) -- node  {} (U4);
			\draw (U4) -- node  {} (U5);
			\draw (U5) -- node  {} (L5);
			\draw (L1) -- node  {} (L2);
			\draw (L2) -- node  {} (L3);
			\draw (L3) -- node  {} (L4);
			\draw (L4) -- node  {} (L5);
			\draw (L1) edge [color = blue,<->, bend right, shorten >=4pt, shorten <=4pt] node  {} (U1);
			\draw (L2) edge [color = blue,<->, bend right, shorten >=4pt, shorten <=4pt] node  {} (U2);
			\draw (L4) edge [color = blue,<->, bend left, shorten >=4pt, shorten <=4pt] node  {} (U4);
			\draw (L5) edge [color = blue,<->, bend left, shorten >=4pt, shorten <=4pt] node  {} (U5);
		\end{scope}
	\end{tikzpicture}
\end{center}
\begin{center}
	\begin{tikzpicture}
		\coordinate (A) at (0,1.5);
		\node at (A) {Diagram \RomanNumeralCaps{4}: Satake diagram of type $A^{(1)}_{2r+2,1}$.};
		\matrix [column sep={0.6cm}, row sep={0.5 cm,between origins}, nodes={draw = none,  inner sep = 3pt}]
		{
			\node(U1) [draw, circle, fill=white, scale=0.6, label = 0] {};
			&\node(U2)[draw, circle, fill=white, scale=0.6, label =1] {};
			&\node(U3) {$\cdots$};
			&\node(U5)[draw, circle, fill=white, scale=0.6, label =$r$] {};
			\\
			&&&&
			\node(R)[draw, circle, fill=white, scale=0.6, label =$r+1$] {};
			\\
			\node(L1) [draw, circle, fill=white, scale=0.6, label =below:$2r+2$] {};
			&\node(L2)[draw, circle, fill=white, scale=0.6, label =below:$2r+1$] {};
			&\node(L3) {$\cdots$};
			&\node(L5)[draw, circle, fill=white, scale=0.6, label =below:$r+2$] {};
			\\
		};
		\begin{scope}
			\draw (U1) -- node  {} (U2);
			\draw (U2) -- node  {} (U3);
			\draw (U3) -- node  {} (U5);
			\draw (U5) -- node  {} (R);
			\draw (U1) -- node  {} (L1);
			\draw (L1) -- node  {} (L2);
			\draw (L2) -- node  {} (L3);
			\draw (L3) -- node  {} (L5);
			\draw (L5) -- node  {} (R);
			\draw (R) edge [color = blue,loop right, looseness=40, <->, shorten >=4pt, shorten <=4pt] node {} (R);
			\draw (L1) edge [color = blue,<->, bend right, shorten >=4pt, shorten <=4pt] node  {} (U1);
			\draw (L2) edge [color = blue,<->, bend right, shorten >=4pt, shorten <=4pt] node  {} (U2);
			\draw (L5) edge [color = blue,<->, bend left, shorten >=4pt, shorten <=4pt] node  {} (U5);
		\end{scope}
	\end{tikzpicture}
\end{center}
\begin{center}
	\begin{tikzpicture}
		\coordinate (A) at (0,1.5);
		\node at (A) {Diagram \RomanNumeralCaps{5}: Satake diagram of type $A^{(1)}_{1,2r+2}$.};
		\matrix [column sep={0.6cm}, row sep={0.5 cm,between origins}, nodes={draw = none,  inner sep = 3pt}]
		{
			&\node(U1) [draw, circle, fill=white, scale=0.6, label = 1] {};
			&\node(U3) {$\cdots$};
			&\node(U4)[draw, circle, fill=white, scale=0.6, label =$r$] {};
			&\node(U5)[draw, circle, fill=white, scale=0.6, label =$r+1$] {};
			\\
			\node(L)[draw, circle, fill=white, scale=0.6, label =0] {};
			&&&&&
			\\
			&\node(L1) [draw, circle, fill=white, scale=0.6, label =below:$2r+2$] {};
			&\node(L3) {$\cdots$};
			&\node(L4)[draw, circle, fill=white, scale=0.6, label =below:$r+3$] {};
			&\node(L5)[draw, circle, fill=white, scale=0.6, label =below:$r+2$] {};
			\\
		};
		\begin{scope}
			\draw (L) -- node  {} (U1);
			\draw (U1) -- node  {} (U3);
			\draw (U3) -- node  {} (U4);
			\draw (U4) -- node  {} (U5);
			\draw (U5) -- node  {} (L5);
			\draw (L) -- node  {} (L1);
			\draw (L1) -- node  {} (L3);
			\draw (L3) -- node  {} (L4);
			\draw (L4) -- node  {} (L5);
			\draw (L) edge [color = blue, loop left, looseness=40, <->, shorten >=4pt, shorten <=4pt] node {} (L);
			\draw (L1) edge [color = blue,<->, bend right, shorten >=4pt, shorten <=4pt] node  {} (U1);
			\draw (L4) edge [color = blue,<->, bend left, shorten >=4pt, shorten <=4pt] node  {} (U4);
			\draw (L5) edge [color = blue,<->, bend left, shorten >=4pt, shorten <=4pt] node  {} (U5);
		\end{scope}
	\end{tikzpicture}
\end{center}
\begin{center}
	
	\begin{tikzpicture}	\label{figure:ii}
		\coordinate (A) at (0,1.5);
		\node at (A) {Diagram \RomanNumeralCaps{6}: Satake diagram of type $A^{(1)}_{2r+1,2}$.};
		\matrix [column sep={0.6cm}, row sep={0.5 cm,between origins}, nodes={draw = none,  inner sep = 3pt}]
		{
			&\node(U1) [draw, circle, fill=white, scale=0.6, label = 1] {};
			&\node(U2) {$\cdots$};
			&\node(U3)[draw, circle, fill=white, scale=0.6, label =$r$] {};
			\\
			\node(L)[draw, circle, fill=white, scale=0.6, label =0] {};
			&&&&
			\node(R)[draw, circle, fill=white, scale=0.6, label =$r+1$] {};
			\\
			&\node(L1) [draw, circle, fill=white, scale=0.6, label =below:$2r+1$] {};
			&\node(L2) {$\cdots$};
			&\node(L3)[draw, circle, fill=white, scale=0.6, label =below:$r+2$] {};
			\\
		};
		\begin{scope}
			\draw (L) -- node  {} (U1);
			\draw (U1) -- node  {} (U2);
			\draw (U2) -- node  {} (U3);
			\draw (U3) -- node  {} (R);
			\draw (L) -- node  {} (L1);
			\draw (L1) -- node  {} (L2);
			\draw (L2) -- node  {} (L3);
			\draw (L3) -- node  {} (R);
			\draw (L) edge [color = blue, loop left, looseness=40, <->, shorten >=4pt, shorten <=4pt] node {} (L);
			\draw (R) edge [color = blue,loop right, looseness=40, <->, shorten >=4pt, shorten <=4pt] node {} (R);
			\draw (L1) edge [color = blue,<->, bend right, shorten >=4pt, shorten <=4pt] node  {} (U1);
			\draw (L3) edge [color = blue,<->, bend left, shorten >=4pt, shorten <=4pt] node  {} (U3);
		\end{scope}
	\end{tikzpicture}
\end{center}

\begin{prop}{\label{iotaast}}
Let $\mathcal S$ be the Satake diagram in Diagram
 \RomanNumeralCaps{1}-\RomanNumeralCaps{6}. Then  there is a $\mathbb Q(q)$-algebra homomorphism $\iota:\mathbf{A}_q(\mathcal S)\rightarrow \mathbf A_q(A_{r+1})$ which sends
	\begin{align*}
		&\mathfrak d_i\mapsto\mathcal D_i(\sum_{k=0}^{\xi_i-1}{\mathcal M_i^{\xi_i-1-2k}}),\quad\mathfrak x_i\mapsto\mathcal X_i,\quad\mathfrak m_i^{\pm 1}\mapsto \mathcal M_i^{\pm \xi_i},\ \ \text{if $\xi_i>0$},\\
		&\mathfrak d_i\mapsto-\mathcal D_i(\sum_{k=0}^{-\xi_i-1}{\mathcal M_i^{-\xi_i-1-2k}}),\quad\mathfrak x_i\mapsto\mathcal X_i,\quad\mathfrak m_i^{\pm 1}\mapsto \mathcal M_i^{\pm \xi_i},\ \ \text{if  $\xi_i<0$}.\\
		\end{align*}
\end{prop}

\begin{proof}
Since $\mathfrak d_i$, $\mathfrak x_i$ and $\mathfrak m_i$ are generators of $\mathbf A_q(\mathcal S)$,
it's enough to show that the map $\iota$ is well defined.
The verification of relations \eqref{mimj}-\eqref{dimiximi} is straightforward.
 We shall only verify the relations in \eqref{dixixidi}. If $\xi_i>0$, by using the relations \eqref{DiMiXiMi} and \eqref{DiXiXiDi}, we have
\allowdisplaybreaks
\begin{align*}
	\iota(\mathfrak d_{i}\mathfrak x_{i})=&\mathcal D_i(\sum_{k=0}^{\xi_i-1}{\mathcal M_i^{\xi_i-1-2k}})\mathcal X_{i}=\mathcal D_i\mathcal X_{i}(\sum_{k=0}^{\xi_i-1}q^{\xi_i-1-2k}{\mathcal M_i^{\xi_i-1-2k}})\\
	=&\frac{q\mathcal M_i-q^{-1}\mathcal M_i^{-1}}{q-q^{-1}}(\sum_{k=0}^{\xi_i-1}q^{\xi_i-1-2k}{\mathcal M_i^{\xi_i-1-2k}})
	=\frac{q^{\xi_i}\mathcal M_i^{\xi_i}-q^{-\xi_i}\mathcal M_i^{-\xi_i}}{q-q^{-1}}\\
	=&\frac{q^{\xi_i}\iota(\mathfrak m_i)-q^{-\xi_i}\iota(\mathfrak m_i^{-1})}{q-q^{-1}},
	\end{align*}
and
\begin{align*}
    \iota(\mathfrak x_i\mathfrak d_i)=&\mathcal X_i\mathcal D_i(\sum_{k=0}^{\xi_i-1}{M_i^{\xi_i-1-2k}})=\frac{\mathcal M_i-\mathcal M_i^{-1}}{q-q^{-1}}(\sum_{k=0}^{\xi_i-1}{\mathcal M_i^{\xi_i-1-2k}})\\
	=&\frac{\sum_{k=0}^{\xi_i-1}\mathcal M_i^{\xi_i-2k}-\sum_{k=0}^{\xi_i-1}\mathcal M_i^{\xi_i-2-2k}}{q-q^{-1}}=\frac{\sum_{k=0}^{\xi_i-1}\mathcal M_i^{\xi_i-2k}-\sum_{k=1}^{\xi_i}\mathcal M_i^{\xi_i-2k}}{q-q^{-1}}\\
	=&\frac{\mathcal M_i^{\xi_i}-\mathcal M_i^{-\xi_i}}{q-q^{-1}}=\frac{\iota(\mathfrak m_i)-\iota(\mathfrak m_i^{-1})}{q-q^{-1}}.
\end{align*}

The case for $\xi_i<0$ can be shown similarly.
\end{proof}

\begin{rem}
The diagram constructed from Diagram \RomanNumeralCaps{1} by neglecting the edge
between $r+1$ and $r+2$ is called
 the Satake diagram of diagonal type. In this case, the modified $q$-Weyl algebra $\mathbf{A}_q(\mathcal S)$ is exactly the $q$-Weyl algebra $\mathbf A_q(A_{r+1})$.
\end{rem}

Let $\mathbf X=(X_0,\cdots,X_{r+1})$ and
$\mathbf a={(a_i)}_{i\in\mathbb I}\in\mathbb Z_{\ge 0}^{r+2}$.
We denote $ \mathbf X^{\mathbf a}=X_0^{a_0}X_1^{a_1}\cdots X_{r+1}^{a_{r+1}}$. Let $\mathbf e_i$ be the $(r+2)$-tuple such that the $i$-th entry is 1 and $0$ otherwise.
We define a lexicographic order $>_{lex}$ on the set  $\mathbb Z^{r+2}_{\ge 0}$ as follows.

For any $\mathbf a=(a_0,\cdots,a_{r+1}),~\mathbf b=(b_0,\cdots,b_{r+1})\in \mathbb Z^{r+2}_{\ge 0}$,  $\mathbf a>_{lex} \mathbf b$ if the leftmost nonzero entry of ${\mathbf a} -{\mathbf b}\in \mathbb Z^{r+2}$ is positive.

\begin{thm}\label{irreducible Aast module}
The polynomial ring $\mathbb P$ is a $\mathbf{A}_q(\mathcal S)$-module which is irreducible if $q$ is not a root of unity.
\end{thm}
\begin{proof}
By Theorem \ref{irreducible Aq module} and Proposition \ref{iotaast}, $\mathbb P$ is a $\mathbf{A}_q(\mathcal S)$-module via pull-back. 
 More precisely, the generators $\mathfrak d_i$, $\mathfrak x_i$ and $\mathfrak m_i$  in $\mathbf{A}_q(\mathcal S)$ act on $\mathbb P$ as follows:
\begin{align}\label{A_j action}
	{\mathfrak d_i}\mathbf X^{\mathbf a}=[\xi_ia_i]\mathbf X^{\mathbf a-\mathbf {e}_i}
	,\quad{\mathfrak x_i}\mathbf X^{\mathbf a}=
	\mathbf X^{\mathbf a+\mathbf e_i}
	,\quad
	{\mathfrak m_i}\mathbf X^{\mathbf a}=
	q^{\xi_ia_i}\mathbf X^{\mathbf a},
\end{align}
where $\mathfrak d_i\mathbf X^{\mathbf a}=0$ if $a_i=0$.

Let $\mathbb P'$ be a nonzero submodule of $\mathbb P$. We choose a non-zero vector  $w=\sum_{\mathbf a}c_{\mathbf a}\mathbf X^{\mathbf a}\in\mathbb P'$ with $c_{\mathbf a}\in \mathbb Q(q)$. 
Let $\mathbf k=(k_i)_{0\leq i\leq r+1}$ be the maximum element in the set $\{\mathbf a\in\mathbb Z^{r+2}_{\geq 0}|c_{\mathbf a}\neq 0\}$ with respect to the order $>_{lex}$. It follows that
\begin{equation*}
\mathfrak d_0^{k_0}	\mathfrak d_1^{k_1}\cdots\mathfrak d_{r+1}^{k_{r+1}}w
=c_{\mathbf k}\prod_{i=0}^{r+1}{[k_i]}^{\xi_i}!\mathbf X^{\mathbf 0}.
\end{equation*}
Since the coefficient of $\mathbf X^{\mathbf 0}$ is not zero, we have $\mathbf X^{\mathbf 0}\in \mathbb P'$. For any element $v=\mathbf X^{\mathbf a}\in \mathbb P$ with $\mathbf a=(a_i)_{0\leq i\leq r+1}$, we have $\mathfrak x_0^{a_0}\mathfrak x_1^{a_1}\cdots \mathfrak x_{r+1}^{a_{r+1}}\mathbf X^{\mathbf 0}=v \in \mathbb P'$, i.e., $\mathbb P\subseteq \mathbb P'$. 
Hence $\mathbb P=\mathbb P'$.
Since $\mathbb P'$ is arbitary, $\mathbb P$ is irreducible.
\end{proof}

\section{Differential operator approach to $\imath$quantum groups}\label{Differential operator realization}

In this section, we shall show that there is an algebra homomorphism  from $^{\imath}U(\mathcal S)$ to $\mathbf A_q(\mathcal S)$ for a quasi-split Satake diagram $\mathcal S$.
In other words, this provides a differential operator approach to $^{\imath}U(\mathcal S)$.
Let us first recall the definition of $^{\imath}U(\mathcal S)$ from \cite{CLW18}.

\allowdisplaybreaks
\begin{Def}
Given a Satake diagram $S$, the associated $\imath$quantum group $^{\imath}U(\mathcal S)$ is the $\mathbb Q(q)$-algebra generated by $B_i$, $H_i$ ($i\in\mathbf I$), subjecting to the following relations 
\begin{align}
	&H_iH_{\tau i} =1,
	\quad
	H_i H_j=H_jH_i,   \label{relation1}
	\\
	&H_j B_i-q^{i\cdot\tau j-i\cdot j} B_iH_j =0, \label{relation2}
	\\
	&B_iB_{j}-B_jB_i =0, \quad \text{ if }i\cdot j =0 \text{ and }\tau i\neq j, \label{relation3}
	\\
	&\sum_{n=0}^{1-i\cdot j} (-1)^nB_i^{(n)}B_jB_i^{(1-i\cdot j-n)} =0, \quad \text{ if } j \neq \tau i\neq i, \label{relation4}
	\\
&\sum_{n=0}^{1-i\cdot\tau i} (-1)^{n+i\cdot\tau i} B_{ i}^{(n)}B_{\tau i}B_{ i}^{(1-i\cdot\tau i-n)}\label{relation5}\\
&\qquad\qquad\qquad=\frac{1}{q-q^{-1}}	      \left(q^{i\cdot\tau i+3(\delta_{i,0}-\delta_{i,\tau 0})\delta_{0\cdot\tau 0,-1}} (q^{-2};q^{-2})_{-i\cdot\tau i} \varsigma_{\tau i}B_{i}^{(-i\cdot\tau i)} H_i  \right.\notag\\
	&\qquad\qquad\qquad \left. -q^{3(\delta_{i,\tau 0}-\delta_{i, 0})\delta_{0\cdot\tau 0,-1}}(q^{2};q^{2})_{-i\cdot\tau i}\varsigma_{i}B_{ i}^{(-i\cdot\tau i)}   H_{\tau i} \right),
	\text{ if } \tau i \neq i,	
\notag \\
	&\sum_{n=0}^{1-i\cdot j} (-1)^n {\begin{bmatrix} 1-i\cdot j\\n \end{bmatrix}} B_i^{n}B_j B_i^{1-i\cdot j-n} =\delta_{i\cdot j,-1}q\varsigma_i B_j,\quad \text{ if }\tau i=i\neq j,
	\label{relation6}
\end{align}
where $\varsigma_i\in {\mathbb Q(q)}^\times$  satisfying $\varsigma_i=\varsigma_{\tau i}$ if $i\cdot\tau i=0$.
\end{Def}

For the Diagram \RomanNumeralCaps{1}-\RomanNumeralCaps{6}, let
\begin{equation*}
	(\varsigma_i, \varsigma_{\tau i})=
	\begin{cases}
		(q^{-1},q^{-1}), &\text{if $i\cdot\tau i=2$},\\
		(1,1), &\text{if $i\cdot\tau i=0$},\\
		(q,1), &\text{if $i\cdot\tau i=-1$},\\
		(q,-q^{-3}), &\text{if $i\cdot\tau i=-2$}.
	\end{cases}	
\end{equation*}


		For $i\in\mathbb I$, let
		\begin{equation*}
			f_i=B_i,\ e_i=B_{\tau i},\ k_i=H_{i},\ k_i^{-1}=H_{\tau i},\ \text{if}\ \tau i\neq i\ \text{and}\
			t_i=B_i,\ \text{if}\ \tau i= i.
		\end{equation*}

\begin{thm} \label {algebra homo}
Given a Satake diagram $S$ in Diagram \RomanNumeralCaps{1}-\RomanNumeralCaps{6},
there exists a $\mathbb Q(q)$-algebra homomorphism $\varPhi_{\mathcal S}:~^{\imath}U(\mathcal S) \rightarrow \mathbf{A}_q(\mathcal S)$.
\end{thm}
\begin{proof}
For each case of $S$, we shall give an explicit construction of $\varPhi_{\mathcal S}$, and show that it is an algebra homomorphism.
\begin{enumerate}
\item
If $S$ is $A_{2r+1}$ ($r\ge0$) in Diagram I, the map $\varPhi_{\mathcal S}$ is given by
	\begin{equation*}
	e_i\mapsto\mathfrak{x}_{i}\mathfrak{d}_{i+1},\quad
	f_i\mapsto\mathfrak{x}_{i+1}\mathfrak{d}_{i},\quad
	k_i\mapsto \mathfrak{m}_{i}\mathfrak{m}_{i+1}^{-1}.
\end{equation*}	
\item \label {algebra homo 2}
If $S$ is $A_{2r+2,1}$ ($r\ge0$) in Diagram II, the map $\varPhi_{\mathcal S}$ is given by
	\begin{equation*}
	e_i\mapsto\mathfrak{x}_{i}\mathfrak{d}_{i+1},\quad
	f_i\mapsto\mathfrak{x}_{i+1}\mathfrak{d}_{i},\quad
	k_i\mapsto {(-1)}^{\delta_{i,r}}\mathfrak{m}_{i}\mathfrak{m}_{i+1}^{{-(-1)}^{\delta_{i,r}}},\quad 	t_{r+1}\mapsto
	\mathfrak{x}_{r+1}\mathfrak{d}_{r+1}.
\end{equation*}	
\item
If $S$ is $A_{2r+1}^{(1)}$ ($r\ge1$) in Diagram III, the map $\varPhi_{\mathcal S}$ is given by
	\begin{equation*}
	e_i\mapsto\mathfrak{x}_{i}\mathfrak{d}_{i+1},\quad
	f_i\mapsto\mathfrak{x}_{i+1}\mathfrak{d}_{i},\quad
	k_i\mapsto q^{-2\delta_{i,0}}\mathfrak{m}_{i}\mathfrak{m}_{i+1}^{-1}.
\end{equation*}	
\item
If $S$ is $A_{1}^{(1)}$ in Diagram III, the map $\varPhi_{\mathcal S}$ is given by
	\begin{equation*}
	e_0\mapsto\mathfrak{x}_{0}\mathfrak{d}_{1},\quad
	f_0\mapsto\mathfrak{x}_{1}\mathfrak{d}_{0},\quad
	k_0\mapsto q^{-1}\mathfrak{m}_{0}\mathfrak{m}^{-1}_{1}.
\end{equation*}
\item
If $S$ is $A_{2r+2,1}^{(1)}$ ($r\ge0$) in Diagram IV, the map $\varPhi_{\mathcal S}$ is given by
	\begin{align*}
	e_i\mapsto\mathfrak{x}_{i}\mathfrak{d}_{i+1},\quad
	f_i\mapsto\mathfrak{x}_{i+1}\mathfrak{d}_{i},\quad
	k_i\mapsto {(-1)}^{\delta_{i,r}}q^{-2\delta_{i,0}}\mathfrak{m}_{i}\mathfrak{m}^{{-(-1)}^{\delta_{i,r}}}_{i+1},\quad
	t_{r+1}\mapsto \mathfrak x_{r+1}\mathfrak d_{r+1}.
\end{align*}	
\item
If $S$ is $A_{1,2r+2}^{(1)}$ ($r\ge0$) in Diagram V,
the map $\varPhi_{\mathcal S}$ is given by
	\begin{equation*}
	e_i\mapsto\mathfrak{x}_{i-1}\mathfrak{d}_{i},\quad
	f_i\mapsto\mathfrak{x}_{i}\mathfrak{d}_{i-1},\quad
	k_i\mapsto {(-1)}^{\delta_{i,1}}\mathfrak{m}_{i-1}^{{(-1)}^{\delta_{i,1}}}\mathfrak{m}^{-1}_{i},\quad
	t_{0}\mapsto \mathfrak x_{0}\mathfrak d_{0}.
\end{equation*}
\item
If $S$ is $A_{2r+1,2}^{(1)}$ ($r\ge1$) in Diagram VI, the map $\varPhi_{\mathcal S}$ is given by
\begin{equation*}
	e_i\mapsto\mathfrak{x}_{i}\mathfrak{d}_{i+1},\quad
	f_i\mapsto\mathfrak{x}_{i+1}\mathfrak{d}_{i},\quad
	k_i\mapsto {(-1)}^{\delta_{i,r}}\mathfrak{m}_{i}\mathfrak{m}^{-{(-1)}^{\delta_{i,r}}}_{i+1},\quad
	t_{0}\mapsto\mathfrak x_1\mathfrak d_1,\quad
	t_{r+1}\mapsto\mathfrak x_{r+1}\mathfrak d_{r+1}.
\end{equation*}		
\end{enumerate}
We shall only prove the cases (1)-(4), the other cases can be shown similarly.

(1) By the relations \eqref{relation2} and \eqref{relation5}, we have
	\begin{align}
		&k_ie_jk_i^{-1}= q^{\texttt{c}_{ij}+\delta_{i,r}\delta_{j,r}} e_j,\quad k_if_jk_i^{-1}= q^{-\texttt{c}_{ij}-\delta_{i,r}\delta_{j,r}} f_j,\label{kek}\\
		& e_{r}^2f_{r} + f_{r}e_{r}^2
		=  (q+q^{-1}) (e_{r}f_{r}e_{r}-e_{r}( qk_{r}+q^{-1}k^{-1}_{r} )),\label{j-serre relation}\\
		&f_{r}^2e_{r} + e_{r}f_{r}^2
		=  (q+q^{-1}) (f_{r}e_{r}f_{r}-( qk_{r}+q^{-1}k^{-1}_{r} )f_{r}).	
	\end{align}
	
	By the involution $\tau$ in Diagram \RomanNumeralCaps{1}, it follows that $\xi_i=2^{\delta_{i,r+1}}$. According to the definition of homomorphism $\varPhi_{\mathcal S}$ and relations \eqref{mimj}-\eqref{dixixidi}, for $i=j= r$ in the relations \eqref{kek}, we have
	\begin{align*}
		&{\varPhi_{\mathcal S}}(k_re_rk_r^{-1})=\mathfrak{m}_{r}\mathfrak{m}_{r+1}^{-1}\mathfrak{x}_{r}\mathfrak{d}_{r+1}\mathfrak{m}_{r}^{-1}\mathfrak{m}_{r+1}=q^3\mathfrak{x}_{r}\mathfrak{d}_{r+1}=q^3{\varPhi_{\mathcal S}}(e_r),\\
		&{\varPhi_{\mathcal S}}(k_rf_rk_r^{-1})=\mathfrak{m}_{r}\mathfrak{m}_{r+1}^{-1}\mathfrak{x}_{r+1}\mathfrak{d}_{r}\mathfrak{m}_{r}^{-1}\mathfrak{m}_{r+1}=q^{-3}\mathfrak{x}_{r+1}\mathfrak{d}_{r}=q^{-3}{\varPhi_{\mathcal S}}(f_r),
	\end{align*}
and the verification of other cases in the relations \eqref{kek} is similar.
	
For the relation \eqref{j-serre relation}, we have
\begin{align*}
	&{\varPhi_{\mathcal S}(e_r^{2}f_r)}+	{\varPhi_{\mathcal S}(f_re_r^{2})}
	 =\mathfrak{x}_{r}\mathfrak{d}_{r+1}\mathfrak{x}_{r}\mathfrak{d}_{r+1}\mathfrak{x}_{r+1}\mathfrak{d}_{r}+\mathfrak{x}_{r+1}\mathfrak{d}_{r}\mathfrak{x}_{r}\mathfrak{d}_{r+1}\mathfrak{x}_{r}\mathfrak{d}_{r+1}\\
	 =&{(q-q^{-1})}^{-2}\mathfrak{x}_{r}\mathfrak{d}_{r+1}((\mathfrak{m}_{r}-\mathfrak{m}_{r}^{-1})(q^2\mathfrak{m}_{r+1}-q^{-2}\mathfrak{m}_{r+1}^{-1})
	+(q^2\mathfrak{m}_{r}-q^{-2}\mathfrak{m}_{r}^{-1})(q^{-2}\mathfrak{m}_{r+1}-q^2\mathfrak{m}_{r+1}^{-1}))\\
	 =&{(q-q^{-1})}^{-2}\mathfrak{x}_{r}\mathfrak{d}_{r+1}((q^2+1)\mathfrak{m}_{r}\mathfrak{m}_{r+1}+(q^{-2}+1)\mathfrak{m}_{r}^{-1}\mathfrak{m}^{-1}_{r+1}\\
	&\qquad\qquad\qquad\qquad-(q^{-2}+q^4)\mathfrak{m}_{r}\mathfrak{m}_{r+1}^{-1}-(q^{2}+q^{-4})\mathfrak{m}_r\mathfrak{m}^{-1}_{r+1})\\
	=&(q+q^{-1}){(q-q^{-1})}^{-2}\mathfrak{x}_{r}\mathfrak{d}_{r+1}(q\mathfrak m_r\mathfrak m_{r+1}+q^{-1}\mathfrak m_r^{-1}\mathfrak m_{r+1}^{-1}\\
	&\qquad\qquad\qquad\qquad-(q+q(q-q^{-1})^2)\mathfrak m_r\mathfrak m_{r+1}^{-1}-(q^{-1}+q^{-1}(q-q^{-1})^2)\mathfrak m_r^{-1}\mathfrak m_{r+1})\\
	=&(q+q^{-1})\mathfrak{x}_{r}\mathfrak{d}_{r+1}({(q-q^{-1})}^{-2}(\mathfrak m_{r+1}-\mathfrak m_{r+1}^{-1})(q\mathfrak m_r-q^{-1}\mathfrak m_r^{-1})-q\mathfrak m_{r}\mathfrak m_{r+1}^{-1}-q^{-1}\mathfrak m_{r}^{-1}\mathfrak m_{r+1})\\
	 =&(q+q^{-1})(\mathfrak{x}_{r}\mathfrak{d}_{r+1}\mathfrak{x}_{r+1}\mathfrak{d}_{r}\mathfrak{x}_{r}\mathfrak{d}_{r+1}-q\mathfrak{x}_{r}\mathfrak{d}_{r+1}\mathfrak m_{r}\mathfrak m_{r+1}^{-1}-q^{-1}\mathfrak{x}_{r}\mathfrak{d}_{r+1}\mathfrak m_{r}^{-1}\mathfrak m_{r+1})\\
	=&(q+q^{-1})(\varPhi_{\mathcal S}(e_rf_re_r)-q\varPhi_{\mathcal S}(e_rk_r)-q^{-1}\varPhi_{\mathcal S}(e_rk_r^{-1})).\\
\end{align*}
(2) For $i,j\in\{r,r+1\}$, the  relations \eqref{relation4} and \eqref{relation6} can be converted to
\begin{align}
	&e_{r}^2 t_{r+1} + t_{r+1} e_{r}^2
	= [2]  e_{r} t_{r+1} e_{r},\label{e2t}\\
	&f_{r}^2 t_{r+1} + t_{r+1} f_{r}^2
	= [2] f_{r} t_{r+1} f_{r}\label{f2t}\\
	&e_i t_{r+1} =t_{r+1} e_i,\quad f_i t_{r+1} = t_{r+1} f_i, \ \ \text{if $i\neq r$},\\
	&t_{r+1}^2 e_{r} + e_{r} t_{r+1}^2 = [2]  t_{r+1} e_{r} t_{r+1} + e_{r},\label{t2e} \\
	&t_{r+1}^2 f_{r} + f_{r} t_{r+1}^2 = [2]t_{r+1} f_{r} t_{r+1} + f_{r},\label{t2f}.
\end{align}

By the involution $\tau$ in Diagram \RomanNumeralCaps{2}, we have  $\xi_i={(-1)}^{\delta_{i,r+1}}$. For the relations \eqref{e2t} and \eqref{t2e}, we have
\begin{align*}
	&\varPhi_{\mathcal S}(e_{r}^2t_{r+1})+\varPhi_{\mathcal S}( t_{r+1}e_{r}^2)\\
	=&\mathfrak x_r^2\mathfrak d_{r+1}^2\frac{\mathfrak m_{r+1}-\mathfrak m_{r+1}^{-1}}{q-q^{-1}}
	+\mathfrak x_r^2\mathfrak d_{r+1}^2
	\frac{q^2\mathfrak m_{r+1}-q^{-2}\mathfrak m_{r+1}^{-1}}{q-q^{-1}}\\
	=&(q+q^{-1})\mathfrak x_r^2\mathfrak d_{r+1}^2\frac{q\mathfrak m_{r+1}-q^{-1}\mathfrak m_{r+1}^{-1}}{q-q^{-1}}
	=[2]\varPhi_{\mathcal S}(e_{r}t_{r+1}e_{r}),\\
	&\varPhi_{\mathcal S}(t_{r+1}^2 e_{r})+
	\varPhi_{\mathcal S}( e_{r}t_{r+1}^2)\\
	=&{(q-q^{-1})}^{-2}({(\mathfrak m_{r+1}-\mathfrak m_{r+1}^{-1})}^2
	+
	{(q^{-1}\mathfrak m_{r+1}-q\mathfrak m^{-1}_{r+1})}^2)\mathfrak x_r\mathfrak d_{r+1}\\
	=&{(q-q^{-1})}^{-2}((1+q^{-2})\mathfrak m^{2}_{r+1}+(1+q^2)\mathfrak m^{-2}_{r+1}-4)\mathfrak x_r\mathfrak d_{r+1}\\
	=&{(q-q^{-1})}^{-2}(q+q^{-1})(\mathfrak m_{r+1}-\mathfrak m^{-1}_{r+1} )(q^{-1}\mathfrak m_{r+1}-q\mathfrak m^{-1}_{r+1} )\mathfrak x_r\mathfrak d_{r+1}+\mathfrak x_r\mathfrak d_{r+1}\\
	=&[2]\varPhi_{\mathcal S}(t_{r+1}e_rt_{r+1})+\varPhi_{\mathcal S}(e_r).	
\end{align*}

(3)
By the relation \eqref{relation5}, we have
\begin{align}
	&  e_i   f_j -   f_j   e_i  = \delta_{ij} \frac{  k_i -   k_i^{-1}}{q - q^{-1}}, \ \ \text{if $(i, j) \neq (0,0), (r,r)$},\label{effe-0r}\\
	&  e_0^2   f_0 +   f_0   e_0^2
	= (q + q^{-1})(   e_0   f_0   e_0 - (q   k_0 + q^{-1}   k_0^{-1})   e_0),\label{e02f0}\\
	&  e_r^2   f_r +   f_r   e_r^2
	= (q + q^{-1})(   e_r   f_r   e_r -   e_r (q   k_r + q^{-1}   k_r^{-1})),\\
	&  f_0^2   e_0   +   e_0   f_0^2
	= (q + q^{-1})  (   f_0   e_0   f_0 -   f_0 (q   k_0 + q^{-1}   k_0^{-1})),\\
	& f_r^2   e_r   +   e_r   f_r^2
	= (q + q^{-1})  (   f_r   e_r   f_r - (q   k_r + q^{-1}   k_r^{-1})  f_r)\label{fr2er}.
\end{align}

We provide a detailed argument for the relation \eqref{e02f0}.
By the involution $\tau$ in Diagram \RomanNumeralCaps{3}, we have $\xi_i={2}^{\delta_{i,0}+\delta_{i,r+1}}$. Then
\begin{align*}	
	&{\varPhi_{\mathcal S}(e_0^{2}f_0)}+	{\varPhi_{\mathcal S}(f_0e_0^{2})}		
	=\mathfrak{x}_{0}\mathfrak{d}_{1}\mathfrak{x}_{0}\mathfrak{d}_{1}\mathfrak{x}_{1}\mathfrak{d}_{0}
	+\mathfrak{x}_{1}\mathfrak{d}_{0}\mathfrak{x}_{0}\mathfrak{d}_{1}\mathfrak{x}_{0}\mathfrak{d}_{1}	\\
	=&{(q-q^{-1})}^{-2}\mathfrak{x}_{0}\mathfrak{d}_{1}((\mathfrak{m}_{0}-\mathfrak{m}_{0}^{-1})(q\mathfrak{m}_{1}-q^{-1}\mathfrak{m}_{1}^{-1})
	+(q^{-1}\mathfrak{m}_{1}-q\mathfrak{m}_{1}^{-1})(q^4\mathfrak{m}_{0}-q^{-4}\mathfrak{m}_{0}^{-1}))\\	
	=&{(q-q^{-1})}^{-2}\mathfrak{x}_{0}\mathfrak{d}_{1}((q+q^3)\mathfrak{m}_{0}\mathfrak{m}_{1}-(q^{-1}+q^5)\mathfrak{m}_{0}\mathfrak{m}^{-1}_{1}
	-(q+q^{-5})\mathfrak{m}_{0}^{-1}\mathfrak{m}_{1}+(q^{-1}+q^{-3})\mathfrak{m}_0^{-1}\mathfrak{m}^{-1}_{1})\\
	 =&{(q-q^{-1})}^{-2}(q+q^{-1})\mathfrak{x}_{0}\mathfrak{d}_{1}(q^2\mathfrak{m}_{0}\mathfrak{m}_{1}+q^{-2}\mathfrak{m}^{-1}_{0}\mathfrak{m}^{-1}_{1}\\
	 &\qquad\qquad\qquad\qquad-(q^2+{(q-q^{-1})}^2q^2)\mathfrak{m}_{0}\mathfrak{m}_{1}^{-1}-(q^{-2}+{(q-q^{-1})}^2q^{-2})\mathfrak{m}_{0}^{-1}\mathfrak{m}_{1})\\
	=&(q+q^{-1})\mathfrak{x}_{0}\mathfrak{d}_{1}({(q-q^{-1})}^{-2}(\mathfrak m_1-\mathfrak m_1^{-1})(q^{2}\mathfrak m_0-q^{-2}\mathfrak m_0^{-1})-(q^2\mathfrak m_0\mathfrak m_1^{-1}+q^{-2}\mathfrak m_0^{-1}\mathfrak m_1))\\
	 =&(q+q^{-1})(\mathfrak{x}_{0}\mathfrak{d}_{1}\mathfrak{x}_{1}\mathfrak{d}_{1}\mathfrak{x}_{0}\mathfrak{d}_{1}-q^2\mathfrak{x}_{0}\mathfrak{d}_{1}\mathfrak{m}_{0}\mathfrak{m}_{1}^{-1}-q^{-2}\mathfrak{x}_{0}\mathfrak{d}_{1}\mathfrak{m}_{0}^{-1}\mathfrak{m}_{1})\\
	=&(q+q^{-1})(	{\varPhi_{\mathcal S}(e_0f_0e_0)}-q	{\varPhi_{\mathcal S}(k_0e_0)}-q^{-1}{\varPhi_{\mathcal S}(k_0^{-1}e_0)}).
\end{align*}

(4)
It follows by the relations \eqref{relation1}, \eqref{relation2} and \eqref{relation5} that
\begin{align}
	&k_0k_0^{-1} =1, \quad \quad   k_0e_0=q^4e_0k_0,  \qquad
	k_0f_0 = q^{-4}f_0k_0,\\
	&e^{3}_0f_0 - [3]e^2_0f_0e_0 + [3]e_0f_0e^{2}_0 -  f_0e^3_0 = [3]! (q-q^{-1})e_0(k_0 - k^{-1}_0)e_0,\label{U_2^c} \\
	&f^{3}_0e_0 - [3]f^2_0e_0f_0 +[3]f_0e_0f^{2}_0 -e_0f^3_0 = -[3]!(q-q^{-1})f_0(k_0-k^{-1}_0)f_0.
\end{align}

We need to make slight modifications for the algebra $\mathbf{A}_q(\mathcal S)$ defined in Definition \ref{modified q-Weyl algebra}.
Let $r=0$ and $\xi_0=1$, $\xi_1=3$. We only verify the relation \eqref{U_2^c} and
we have
\begin{align*}
	&\varPhi_{\mathcal S}(e_0^{3}f_0)-[3]\varPhi_{\mathcal S}(e_0^{2}f_0e_0)+[3]\varPhi_{\mathcal S}(e_0f_0e_0^{2})-
	\varPhi_{\mathcal S}(f_0e_0^{3})\\
	=&\mathfrak x_{0}\mathfrak d_{1}\mathfrak x_{0}\mathfrak d_{1}\mathfrak x_{0}\mathfrak d_{1}\mathfrak x_{1}\mathfrak d_{0}-[3]\mathfrak x_{0}\mathfrak d_{1}\mathfrak x_{0}\mathfrak d_{1}\mathfrak x_{1}\mathfrak d_{0}\mathfrak x_{0}\mathfrak d_{1}+[3]	\mathfrak x_{0}\mathfrak d_{1}\mathfrak x_{1}\mathfrak d_{0}\mathfrak x_{0}\mathfrak d_{1}\mathfrak x_{0}\mathfrak d_{1}-\mathfrak x_{1}\mathfrak d_{0}\mathfrak x_{0}\mathfrak d_{1}\mathfrak x_{0}\mathfrak d_{1}\mathfrak x_{0}\mathfrak d_{1}\\
	=&{(q-q^{-1})}^{-2}\mathfrak x^2_{0}\mathfrak d^2_{1}((\mathfrak m_0-\mathfrak m_0^{-1})(q^3\mathfrak m_1-q^{-3}\mathfrak m_1^{-1})
	-[3](\mathfrak m_1-\mathfrak m_1^{-1})(q\mathfrak m_0-q^{-1}\mathfrak m_0^{-1}) \\
	&\qquad\qquad+[3](q^{2}\mathfrak m_0-q^{-2}\mathfrak m_0^{-1})(q^{-3}\mathfrak m_1-q^{3}\mathfrak m_1^{-1})
	-(q^{-6}\mathfrak m_1-q^{6}\mathfrak m_1^{-1})(q^{3}\mathfrak m_0-q^{-3}\mathfrak m_0^{-1}))\\
	=&{(q-q^{-1})}^{-2}\mathfrak x^2_{0}\mathfrak d^2_{1}
	(((q^3-q^{-3})+[3](q^{-1}-q))\mathfrak {m}_0\mathfrak {m}_1-((q^{-3}-q^{9})+[3](q^{5}-q))\mathfrak {m}_0\mathfrak {m}_1^{-1}\\
	&\qquad\qquad-((q^{3}-q^{-9})+[3](q^{-5}-q^{-1}))\mathfrak {m}_1\mathfrak {m}_0^{-1}-((q^{-3}-q^{3})+[3](q-q^{-1}))\mathfrak {m}_0^{-1}\mathfrak {m}_1^{-1})\\
	=&[3]!(q-q^{-1})\mathfrak x^2_{0}\mathfrak d^2_{1}(q^3\mathfrak m_0\mathfrak m_1^{-1}-q^{-3}\mathfrak m^{-1}_0\mathfrak m_1)	
	=[3]!(q-q^{-1})\mathfrak x_{0}\mathfrak d_{1}(q^{-1}\mathfrak m_0\mathfrak m_1^{-1}-q\mathfrak m^{-1}_0\mathfrak m_1)\mathfrak x_{0}\mathfrak d_{1}\\
	=&[3]!(q-q^{-1})
	\varPhi_{\mathcal S}(e_0(k_0-k_0^{-1})e_0).
\end{align*}

The other relations in $^{\imath}{U(\mathcal S)}$ can be verified in a similar way.		
\end{proof}

\section{Crystal basis of the oscillator representations}\label{Crystal basis}

In this section, we study the  irreducible modules of $\imath$quantum groups and construct the crystal basis of these irreducible modules over $\imath$quantum groups $^{\imath}U(A_{2r+1})$, $^{\imath} U(A_{2r+1}^{(1)})$ and $^{\imath} U(A_{1}^{(1)})$.

\subsection{Irreducible modules of $\imath$quantum groups}

 By Proposition \ref{iotaast} and Theorem  \ref{algebra homo}, we have

\begin{cor}
	There exists a $\mathbb Q(q)$-algebra homomorphism $\iota\circ\varPhi_{\mathcal S}:  ~^{\imath}U(\mathcal S)\to\mathbf A_q(A_{r+1})$.
\end{cor}

From Theorems \ref{irreducible Aast module} and \ref{algebra homo}, it's easy to see that $\mathbb P$ is a $^{\imath} U(\mathcal S)$-module. The $ ^{\imath} U(\mathcal S)$-action is given as follows.

\begin{itemize}
	\item[1)] For the $\imath$quantum group $^{\imath}U(A_{2r+1})$ ($r\ge0$), the action is given by
	\begin{equation*}
		e_i\mathbf X^{\mathbf a}:=[2^{\delta_{i,r}}a_{i+1}]\mathbf X^{\mathbf a+ \mathbf e_i- \mathbf e_{i+1}},
		\quad f_i\mathbf X^{\mathbf a}:=[a_i]\mathbf X^{\mathbf a-\mathbf e_i+\mathbf e_{i+1}},
		\quad k_i\mathbf X^{\mathbf a}:=q^{a_i-2^{\delta_{i,r}}a_{i+1}}\mathbf X^{\mathbf a}.
	\end{equation*}
	
	\item[2)]
	For the $\imath$quantum group $^{\imath}U(A_{2r+2,1})$ ($r\ge 0$), the action is given by
	\begin{align*}
		&e_i\mathbf X^{\mathbf a}:=[{(-1)}^{\delta_{i,r}}a_{i+1}]\mathbf X^{\mathbf a+ \mathbf e_i-\mathbf e_{i+1}},
		\quad f_i\mathbf X^{\mathbf a}:=[a_{i}]\mathbf X^{\mathbf a-\mathbf e_i+\mathbf e_{i+1}},\\
		& k_i\mathbf X^{\mathbf a}:={(-1)}^{\delta_{i,r}}q^{a_i-a_{i+1}}\mathbf X^{\mathbf a},
		\quad t_{r+1}\mathbf X^{\mathbf a}:=-[a_{r+1}]\mathbf X^{\mathbf a}.
	\end{align*}
	
	\item[3)]
	For the $\imath$quantum group $^{\imath}U(A_{2r+1}^{(1)})$ ($r\ge1$), the action is given by
	\begin{align*}
		&e_i\mathbf X^{\mathbf a}:=[2^{\delta_{i,r}}a_{i+1}]\mathbf X^{\mathbf a+\mathbf e_i-\mathbf e_{i+1}},
		\quad f_i\mathbf X^{\mathbf a}:=[2^{\delta_{i,0}}a_i]\mathbf X^{\mathbf a-\mathbf e_i+\mathbf e_{i+1}},\\
		&k_i\mathbf X^{\mathbf a}:=q^{2^{\delta_{i,0}}a_i-2^{\delta_{i,r}}a_{i+1}-2\delta_{i,0}}\mathbf X^{\mathbf a},
	\end{align*}

     \item[4)]
     For the $\imath$quantum group $^{\imath}U(A_{1}^{(1)})$, the action is given by
	\begin{align*}
		e_0\mathbf X^{\mathbf a}:=[3a_1]\mathbf X^{\mathbf a+\mathbf e_0-\mathbf e_1},\quad
		f_0\mathbf X^{\mathbf a}:=[a_0]\mathbf X^{\mathbf a-\mathbf e_0+\mathbf e_1},\quad
		k_0\mathbf X^{\mathbf a}:=q^{a_0-3a_1-1}\mathbf X^{\mathbf a}.
	\end{align*}
	
	\item[5)]
	For the $\imath$quantum group $^{\imath}U(A_{2r+2,1}^{(1)})$ ($r\ge0$), the action is given by
	\begin{align*}
		&e_i\mathbf X^{\mathbf a}:=
		{(-1)}^{\delta_{i,r}}[a_{i+1}]
		\mathbf X^{\mathbf a+\mathbf e_i-\mathbf e_{i+1}},\quad f_i\mathbf X^{\mathbf a}:=[2^{\delta_{i,0}}a_i]\mathbf X^{\mathbf a-\mathbf e_i+\mathbf e_{i+1}},\\
		& k_i\mathbf X^{\mathbf a}:={(-1)}^{\delta_{i,r}}
		q^{2^{\delta_{i,0}}a_i-a_{i+1}-2\delta_{i,0}}\mathbf X^{\mathbf a},\quad t_{r+1}\mathbf X^{\mathbf a}:=-[a_{r+1}]\mathbf X^{\mathbf a}.
	\end{align*}
	
	\item[6)]
	For the $\imath$quantum group $^{\imath}U(A_{1,2r+2}^{(1)})$ ($r\ge0$), the action is given by
	\begin{align*}
		&e_i\mathbf X^{\mathbf a}:=
		[2^{\delta_{i,r+1}}a_{i}]
		\mathbf X^{\mathbf a+\mathbf e_{i-1}-\mathbf e_{i}},\quad f_i\mathbf X^{\mathbf a}:=[{(-1)}^{\delta_{i,1}}a_{i-1}]\mathbf X^{\mathbf a-\mathbf e_{i-1}+\mathbf e_i},\\
		& k_i\mathbf X^{\mathbf a}:={(-1)}^{\delta_{i,1}}
		q^{a_{i-1}-2^{\delta_{i,r+1}}a_i}\mathbf X^{\mathbf a},\quad t_0\mathbf X^{\mathbf a}:=-[a_0]\mathbf X^{\mathbf a}.
	\end{align*}
	
	\item[7)]
	For the $\imath$quantum group $^{\imath}U(A_{2r+1,2}^{(1)})$ ($r\ge1$), the action is given by
	\begin{align*}
		&e_i\mathbf X^{\mathbf a}:=
		{(-1)}^{\delta_{i,r}}[a_{i+1}]
		\mathbf X^{\mathbf a+\mathbf e_i-\mathbf e_{i+1}},\quad f_i\mathbf X^{\mathbf a}:=[a_i]\mathbf X^{\mathbf a-\mathbf e_i+\mathbf e_{i+1}},\\
		& k_i\mathbf X^{\mathbf a}:=
		{(-1)}^{\delta_{i,r}}q^{a_i-a_{i+1}}\mathbf X^{\mathbf a},\quad
		t_{0}\mathbf X^{\mathbf a}:=[a_{1}]\mathbf X^{\mathbf a},\quad t_{r+1}\mathbf X^{\mathbf a}:=-[a_{r+1}]\mathbf X^{\mathbf a}.	
	\end{align*}
\end{itemize}

The $^{\imath} U(\mathcal S)$-module $\mathbb P$ is called \textit{the oscillator representation} of $^{\imath} U(\mathcal S)$. 
The following commutative diagram explains the relations between $^{\imath} U(\mathcal S)$, $\mathbf{A}_q(\mathcal S)$, $\mathbf A_{q}(A_{r+1})$ and $\mathbb P$.

$$
\xymatrix{
	^{\imath}U(\mathcal S)\ar[rd]\ar[r]^{\varPhi_{\mathcal S}}&\mathbf{A}_q(\mathcal S)\ar[d]\ar[r]^{\iota}&\mathbf A_{q}(A_{r+1})\ar[ld]\\
	&{\rm End}(\mathbb P)&\\
}
$$

For $s\ge 0$, let
$
\Lambda_s=\{\mathbf a\in\mathbb Z_{\ge 0}^{r+2}|\sum_{i=0}^{r+1}a_i=s\}$. The polynomial ring $\mathbb P$ has the following direct sum decomposition
\begin{equation*}
	\mathbb P=\bigoplus_{s=0}^{\infty}{\mathbb P_s},\quad \mathbb P_s=\bigoplus_{\mathbf a\in \Lambda_s}{\mathbb{K}\mathbf X^{\mathbf a}}.
\end{equation*}

Since the actions of $e_i$, $f_i$, $k_i$ and $t_j$ on $\mathbb P$ preserve the degree of arbitrary monomial in $\mathbb P_s$, the subspace $\mathbb P_s$ is a $^{\imath} U(\mathcal S)$-module.

\begin{thm}\label{Uast irreducible module}
	The $^{\imath}U(\mathcal S)$-module $\mathbb P_s$ is irreducible if $q$ is not a root of unity.
\end{thm}
\begin{proof}
	Let $\mathbb P_s'$ be a non-zero submodule of $\mathbb P_s$. We fix a non-zero element
	$\mathbf X^{\mathbf a}\in \mathbb P_s'$. It follows that
	\begin{align*}
		(\prod_{i=0}^{r}{e_i}^{\sum_{j=i}^{r}a_{j+1}}) \mathbf X^{\mathbf a}
		=&(\prod_{i=0}^{r}{(\mathfrak{x}_i\mathfrak{d}_{i+1})}^{\sum_{j=i}^{r}a_{j+1}}) \mathbf X^{\mathbf a}
		=(\prod_{i=1}^{r+1}{[a_i+a_{i+1}+\cdots+a_{r+1}]}^{\xi_i}!)X_0^{s}\in \mathbb P_s'.
	\end{align*}
	
	Since the coefficient is not zero, we get $X_0^{s}\in \mathbb P_s'$. Let $\mathbf X^{\mathbf b}$  be any monomial in $ \mathbb P$, where $\mathbf b=(b_i)_{0\leq i\leq r+1}$. It follows that
	\begin{align*}
		(\prod_{i=0}^rf_i^{s-\sum_{j=0}^ib_j})X_0^{s}
		=&(\prod_{i=0}^{r}{(\mathfrak{x}_{i+1}\mathfrak{d}_{i})}^{s-\sum_{j=0}^ib_j}) X_0^{s}
		=(\frac{{[s]}^{\xi_0}!}{{[b_0]}^{\xi_0}!}\prod_{i=1}^r\frac{{[s-(\sum_{j=0}^{i-1}b_j)]}^{\xi_i}!}{{[b_i]}^{\xi_i}!})\mathbf X^{\mathbf b}\in \mathbb P'_s.
	\end{align*}
	
	The relation $\mathbb P_s= \mathbb P_s'$ holds automatically due to the non-zero coefficient of $\mathbf X^{\mathbf b}$. Hence the $^{\imath}U(\mathcal S)$-module $\mathbb P_s$ is irreducible.
\end{proof}

\subsection{Crystal basis of irreducible module}
In this subsection, we only consider the $\imath$quantum groups $^{\imath}U(A_{2r+1})$, $^{\imath} U(A_{2r+1}^{(1)})$ and $^{\imath} U(A_{1}^{(1)})$.

Let $\mathbf X^{(\mathbf a)}=\prod_{i=0}^{r+1}X_i^{{(a_i)}_{\xi_i}}$ for a sequence $\mathbf a\in\Lambda_s$. It's easy to verify:
\begin{equation*}
	\mathbf X^{(\mathbf a)}=f_i^{{(a_{i+1})}_{\xi_{i+1}}}\mathbf X^{(\mathbf a+a_{i+1}(\mathbf e_i-\mathbf e_{i+1}))}.
\end{equation*}

Let $\mathbb P_s=\bigoplus_{\mathbf a\in\Lambda_s}\mathbb P_s^{\mathbf a}$, where $\mathbb P_s^{\mathbf a}=\mathbb Q(q) \mathbf X^{(\mathbf a)}$. We define Kashiwara operators $\widetilde{e}_i$ and $\widetilde{f}_i$ ($0\leq i\leq r$) on $\mathbb P_s$ as follows:
\begin{equation}\label{Kashiwara operators of Uj}
	\widetilde{e}_i\mathbf X^{(\mathbf a)}=f_i^{{(a_{i+1}-1)}_{\xi_{i+1}}}\mathbf X^{(\mathbf a+a_{i+1}(\mathbf e_i-\mathbf e_{i+1}))},
	\quad\widetilde{f}_i\mathbf X^{(\mathbf a)}=f_i^{{(a_{i+1}+1)}_{\xi_{i+1}}}\mathbf X^{(\mathbf a+a_{i+1}(\mathbf e_i-\mathbf e_{i+1}))}.
\end{equation}

Let  $\mathbb A_0 =\{ f/g \in \mathbb Q(q)|f,g \in \mathbb Q[q], g(0) \neq 0 \}$. The crystal lattice and crystal basis of the irreducible module $\mathbb P_s$ are defined as follows.

\begin{Def}\label{crystal lattice}
	Let $\mathcal L(s)$ be an $\mathbb A_0$-submodule of $\mathbb P_s$. We say that $\mathcal L(s)$ is a crystal lattice of $\mathbb P_s$ if
	\begin{enumerate}
		\item[$(L1)$] $\mathcal L(s)$ is a free $\mathbb A_0$-module of rank $\dim_{\mathbb Q(q)} \mathbb P_s$, and $\mathbb Q(q) \otimes_{\mathbb A_0} \mathcal L(s) = \mathbb P_s$,
		\item[$(L2)$] $\mathcal L(s) = \bigoplus_{\mathbf a \in \Lambda_s} \mathcal L(s)_{\mathbf a}$, where $\mathcal L(s)_{\mathbf a} := \mathcal L(s) \cap \mathbb P_s^{\mathbf a}$,
		\item[$(L3)$] $\widetilde{f}_i(\mathcal L(s)) \subset \mathcal L(s)$ and $\widetilde{e}_i(\mathcal L(s)) \subset \mathcal L(s)$.
	\end{enumerate}
\end{Def}

\begin{Def}\normalfont
	A crystal basis of the module $\mathbb P_s$ is a pair ($\mathcal L(s),\mathcal B(s)$) such that
	\begin{enumerate}
		\item[$(B1)$] $\mathcal L(s)$ is a crystal lattice of $\mathbb P_s$,
		\item[$(B2)$] $\mathcal B(s)$ is a $\mathbb Q$-basis of $\mathcal L(s)/q\mathcal L(s)$,
		\item[$(B3)$] $\mathcal B(s) = \bigsqcup_{\lambda \in \Lambda_s} \mathcal B(s)_\lambda$, where $\mathcal B(s)_\lambda := \mathcal B(s) \cap (\mathcal L(s)_\lambda / q\mathcal L(s)_\lambda)$,
		\item[$(B4)$] $\widetilde{f}_i(\mathcal B(s)) \subset \mathcal B(s) \sqcup \{0\}$ and $\widetilde{e}_i(\mathcal B(s)) \subset \mathcal B(s) \sqcup \{0\}$,
		\item[$(B5)$] for each $b,b' \in \mathcal B(s)$, one has $\widetilde{f}_i(b) = b'$ if and only if $b = \widetilde{e}_i(b')$.
	\end{enumerate}
\end{Def}

From Theorem \ref{Uast irreducible module},  $\mathbb P_s$ is an irreducible $^{\imath} U(\mathcal S)$-module with $\mathbb Q(q)$-basis $\{\mathbf X^{(\mathbf a)}|\mathbf a\in \Lambda_s\}$. We define
\begin{equation}\label{crystal lattice and basis of Uj}
	\mathcal L^{\ast}(s)=\bigoplus_{\mathbf a\in\Lambda_s}{\mathcal L^{\ast}(s)}_{\mathbf a},\quad \mathcal B^{\ast}(s)=\{\mathbf X^{(\mathbf a)}+q\mathcal L^{\ast}(s)|\mathbf a\in\Lambda_s\},
\end{equation}
where ${\mathcal L^{\ast}(s)}_{\mathbf a}=\mathbb A_0 \mathbf X^{(\mathbf a)}$.

\begin{thm}\label{crystal basis of Uj}
	The pair ($\mathcal L^{\ast}(s),\mathcal B^{\ast}(s)$) is a crystal basis of $\mathbb P_s$.	
\end{thm}

\begin{proof}
	It's straightforward to verify that the conditions $(L1)$, $(L2)$ and $(L3)$ hold in Definition \ref{crystal lattice} due to the definition of Kashiwara operators in \eqref{Kashiwara operators of Uj}. Hence $\mathcal L^{\ast}(s)$ is a crystal lattice of $\mathbb P_s$. From the definition of $\mathcal B^{\ast}(s)$ in \eqref{crystal lattice and basis of Uj}, it's easy to see $\mathcal B^{\ast}(s)$ is a $\mathbb Q$-basis of $\mathcal L^{\ast}(s)/q\mathcal L^{\ast}(s)\cong\mathbb Q\otimes_{\mathbb A_0} \mathcal L^{\ast}(s)$. The conditions $(B3)$ and $(B4)$ are easy to verify. Now let us verify $(B5)$:
	
	Let $b=\mathbf X^{(\mathbf a)}$ and $b'=f_i^{{(a_{i+1}+1)}_{\xi_{i+1}}}\mathbf X^{(\mathbf a+a_{i+1}(\mathbf e_i-\mathbf e_{i+1}))}$. Then we have $\widetilde{f}_ib=b'$.
	\begin{align*}
		\widetilde{e}_ib'=&\widetilde{e}_i f_i^{{(a_{i+1}+1)}_{\xi_{i+1}}}\mathbf X^{(\mathbf a+a_{i+1}(\mathbf e_i-\mathbf e_{i+1}))}=\widetilde{e}_i\frac{f_i}{[\xi_{i+1}(a_{i+1}+1)]} f_i^{{(a_{i+1})}_{\xi_{i+1}}}\mathbf X^{(\mathbf a+a_{i+1}(\mathbf e_i-\mathbf e_{i+1}))}\\
		=&\widetilde{e}_i\frac{\mathbf x_{i+1}\mathbf d_i}{[\xi_{i+1}(a_{i+1}+1)]} \mathbf X^{(\mathbf a)}=\widetilde{e}_i\mathbf X^{(\mathbf a-\mathbf e_i+ \mathbf e_{i+1})}=f_i^{{(a_{i+1})}_{\xi_{i+1}}}\mathbf X^{(\mathbf a-\mathbf e_i+\mathbf e_{i+1}+(a_{i+1}+1)(\mathbf e_i-\mathbf e_{i+1}))}\\
		=&\frac{f_i^{a_{i+1}}}{{[a_{i+1}]}^{\xi_{i+1}}!}\mathbf X^{(\mathbf a+a_{i+1}(\mathbf e_i-\mathbf e_{i+1}))}=\mathbf X^{(\mathbf a)}=b.
	\end{align*}
	
	Hence ($\mathcal L^{\ast}(s),\mathcal B^{\ast}(s)$) is a crystal basis of $\mathbb P_s$.
\end{proof}

\begin{ex}
	Let $r=1$ and $s=3$, we use $(a_0a_1a_2)$ to represent $X_0^{{(a_0)}_{\xi_0}}X_1^{{(a_1)}_{\xi_1}}X_2^{{(a_2)}_{\xi_2}}+q\mathcal L^\ast(3)$. We use the red arrow to represent the action of $\widetilde{f}_0$ and the blue arrow to represent the action of $\widetilde{f}_1$. The crystal graph is described as follows:
	\begin{center}
		\begin{tikzpicture}[xscale=1.5,yscale=1.35]
			\node at (0,6) (L6) {$(300)$};
			\node at (1,5) (L5) {$(210)$};
			\node at (0,4) (L4a) {$(201)$};
			\node at (2,4) (L4b) {$(120)$};
			\node at (1,3) (L3a) {$(111)$};
			\node at (3,3) (L3b) {$(030)$};
			\node at (0,2) (L2a) {$(102)$};
			\node at (2,2) (L2b) {$(021)$};
			\node at (1,1) (L1) {$(012)$};
			\node at (0,0) (L0) {${(003)}$};
			
			\draw[thick,->,red  ] (L6) -- (L5);
			\draw[thick,->,red  ] (L5) -- (L4b);
			\draw[thick,->,red  ] (L4b) -- (L3b);
			\draw[thick,->,red  ] (L4a) -- (L3a);
			\draw[thick,->,red  ] (L3a) -- (L2b);
			\draw[thick,->,red  ] (L2a) -- (L1);

			\draw[thick,->,blue ] (L5) -- (L4a);
			\draw[thick,->,blue ] (L4b) -- (L3a);
			\draw[thick,->,blue ] (L3a) -- (L2a);
			\draw[thick,->,blue ] (L3b) -- (L2b);
			\draw[thick,->,blue ] (L2b) -- (L1);
			\draw[thick,->,blue ] (L1) -- (L0);
		\end{tikzpicture}
	\end{center}
\end{ex}

\end{document}